\documentclass[letterpaper,10pt,reqno]{amsart}

\usepackage[backend=biber,style=numeric,giveninits=true,maxnames=4,minnames=4]{biblatex}
\usepackage[portrait,margin=3cm]{geometry}
\usepackage{mathtools}
\mathtoolsset{showonlyrefs}
\usepackage{bbm}
\usepackage{enumitem}

\usepackage{xcolor}
\definecolor{MyDarkblue}{rgb}{0,0.08,0.50}
\definecolor{Brickred}{rgb}{0.65,0.08,0}
\usepackage{tikz}
\usetikzlibrary{positioning}

\usepackage{hyperref}
\hypersetup{
    colorlinks=true,
    linkcolor=MyDarkblue,
    citecolor=Brickred,
    filecolor=red,
    urlcolor=cyan
}

\newtheorem{theorem}{Theorem}[section]
\newtheorem{lemma}[theorem]{Lemma}
\newtheorem{proposition}[theorem]{Proposition}
\newtheorem{corollary}[theorem]{Corollary}

\theoremstyle{definition}
\newtheorem{definition}[theorem]{Definition}

\newcommand{\prob}{\mathbb{P}}

\newcommand{\eps}{\varepsilon}
\newcommand{\N}{\mathbb{N}}
\newcommand{\R}{\mathbb{R}}
\newcommand{\Var}{\mathrm{Var}}
\newcommand{\cB}{\mathcal{B}}
\newcommand{\cO}{\mathcal{O}}
\newcommand{\cP}{\mathcal{P}}
\newcommand{\cR}{\mathcal{R}}
\newcommand{\cS}{\mathcal{S}}
\newcommand{\ind}{\mathbbm{1}}
\newcommand{\ceil}[1]{\lceil #1\rceil}
\newcommand{\Pv}{\mathbb{P}}
\def\toinp{\buildrel{\Pv}\over{\longrightarrow}}
\def\toindis{\buildrel{d}\over{\longrightarrow}}

\numberwithin{equation}{section}

\title[Transition in the asymptotics of joint degrees in uniform recursive DAGs]{Transition in the asymptotics of the joint degree distribution of uniform recursive directed acyclic graphs}
\author{Rafael Engel}
\address{University of Augsburg, Institute of Mathematics, D-86135 Augsburg, Germany}
\email{rafael.engel@uni-a.de}
\date{\today} 

\begin{document}
	
\begin{abstract}
    We study the asymptotic joint degree distribution of uniformly chosen vertices in uniform recursive directed acyclic graphs, a natural generalization of random recursive trees. Here, a graph is built by adding vertices one-by-one and connecting each new vertex to $m\in\N$ uniformly selected vertices already present. Using a representation based on Kingman's coalescent, we establish a transition in the joint degree distribution at the threshold $(m+1)\log n$, arising from the interplay between geometric degree asymptotics and fluctuations in the coalescent construction. Our results extend previously known asymptotics toward this threshold and provide a unified description of the transition, including mixed regimes where prescribed degrees lie on different scales.
\end{abstract}

\maketitle

\section{Introduction}

Random recursive graph models provide a natural framework for networks that grow over time by the successive arrival of new vertices. We refer to Van der Hofstad~\cite{Hofstad.2017,Hof24} and the references therein for a general introduction to evolving random graphs. In this note, we consider the uniform recursive directed acyclic graph (URD), introduced by Devroye and Lu~\cite{Devroye.Lu.1995}. Fix $m\in\N$. Starting from a single vertex labeled $1$, vertices $2,3,\ldots$ are added sequentially. When vertex $i$ is added, it chooses $m\wedge(i-1)$ distinct vertices uniformly at random from those already present and connects to each of them by an outgoing edge. For $m=1$, the resulting graph is the random recursive tree (RRT).
    
The purpose of this note is to describe a transition in the asymptotics of the joint degree distribution of URDs. Our main tool is a representation of the URD using the Kingman coalescent. The Kingman coalescent provides a time-reversed construction that has proved particularly useful for studying the asymptotic behavior of properties such as the degree and depth of multiple vertices in RRTs, see~\cite{Addario.Eslava.2018,Eslava.2017,Eslava.2021,Lodewijks.2023}. It has also been adapted to the construction of RRTs and URDs with a freezing mechanism in~\cite{Bellin.2023,Bellin.2024,EL26}. Here, we work with the corresponding construction for URDs without freezing. In a URD with $n\in\N$ vertices, for degrees of at least a fixed proportion below $(m+1)\log n$,~\cite{EL26} establishes geometric asymptotics for the joint degree distribution of uniformly chosen vertices. We extend these asymptotics to degrees more than $\sqrt{\log n}$ below the threshold $(m+1)\log n$, albeit with weaker error bounds. In a critical window of width $\sqrt{\log n}$ around this threshold, we show that the geometric contribution remains present but acquires additional Gaussian tail factors. Above this window, the joint degree distribution becomes negligible relative to the geometric contribution whenever at least one degree threshold exceeds $(m+1)\log n$ by more than $\sqrt{\log n}$ in asymptotic order. If this excess is at least $\eps\log n$ for some fixed $\eps>0$, we further obtain a polynomial suppression factor. Together, these results identify $(m+1)\log n$ as the threshold beyond which the geometric asymptotics no longer fully describe the degree distribution.

\section{Model definition and results}
\label{sec:main_results}

Let $\N\coloneq\{1,2,\ldots\}$ denote the natural numbers, $\N_0\coloneq \N\cup\{0\}$, and $[n]\coloneq\{1,\ldots,n\}$ for $n\in\N$. In this section, we introduce the uniform recursive directed acyclic graph model and state our main result on the transition in the asymptotics of its joint degree distribution around the scale $(m+1)\log n$.

\begin{definition}[Uniform recursive directed acyclic graph (URD)]
    \label{def:dag}
    Fix $m\in\N$. We recursively construct a sequence $(G^{(n)})_{n\in \N}$ of random directed acyclic graphs. We initialize $G^{(1)}$ as the graph consisting of the single vertex labeled $1$. Conditionally on $G^{(n)}$, we construct $G^{(n+1)}$ by adding a vertex labeled $n+1$, choosing $m\wedge n$ distinct vertices $i_1,\ldots,i_{m\wedge n}$ uniformly at random from $[n]$, and adding the directed edges $(n+1,i_1),\ldots,(n+1,i_{m\wedge n})$.
\end{definition}

\begin{figure}
    \centering
    \resizebox{\textwidth}{!}{
    \begin{tikzpicture}[
        vertex/.style={minimum size=11pt, inner sep=0pt, circle, draw},
        edge/.style={->, >=stealth, draw, thick},
        edgex/.style={->, >=stealth, draw, thick, dotted},
        label/.style={inner sep=0pt},
        scale=1,
        transform shape
    ]

    %%%%%%%%%%%%%%%%%%%%%%%%%%%%%%%%%%%%%%
    % G^(1)
    %%%%%%%%%%%%%%%%%%%%%%%%%%%%%%%%%%%%%%
    \begin{scope}[xshift=0cm, yshift=0cm]
        \node[label] at (-0.8,1.75) {$G^{(1)}$:};

        \node[vertex] (a1) at (0,1.2) {1};
    \end{scope}

    %%%%%%%%%%%%%%%%%%%%%%%%%%%%%%%%%%%%%%
    % G^(2)
    %%%%%%%%%%%%%%%%%%%%%%%%%%%%%%%%%%%%%%
    \begin{scope}[xshift=3.3cm, yshift=0cm]
        \node[label] at (-0.9,1.75) {$G^{(2)}$:};

        \node[vertex] (b1) at (0,1.2) {1};
        \node[vertex] (b2) at (-0.6,0) {2};

        \draw[edgex] (b2) -- (b1);
    \end{scope}

    %%%%%%%%%%%%%%%%%%%%%%%%%%%%%%%%%%%%%%
    % G^(3)
    %%%%%%%%%%%%%%%%%%%%%%%%%%%%%%%%%%%%%%
    \begin{scope}[xshift=7.5cm, yshift=0cm]
        \node[label] at (-1.2,1.75) {$G^{(3)}$:};

        \node[vertex] (c1) at (-0.6,1.2) {1};
        \node[vertex] (c3) at (0,0) {3};
        \node[vertex] (c2) at (-1.2,0) {2};

        \draw[edge]  (c2) -- (c1);
        \draw[edgex] (c3) -- (c1);
        \draw[edgex] (c3) -- (c2);
    \end{scope}

    %%%%%%%%%%%%%%%%%%%%%%%%%%%%%%%%%%%%%%
    % G^(4)
    %%%%%%%%%%%%%%%%%%%%%%%%%%%%%%%%%%%%%%
    \begin{scope}[xshift=12.5cm, yshift=0cm]
        \node[label] at (-1.2,1.75) {$G^{(4)}$:};

        \node[vertex] (d1) at (-0.6,1.2) {1};
        \node[vertex] (d3) at (0,0) {3};
        \node[vertex] (d2) at (-1.2,0) {2};
        \node[vertex] (d4) at (1.2,0) {4};

        \draw[edge]  (d2) -- (d1);
        \draw[edge]  (d3) -- (d1);
        \draw[edge]  (d3) -- (d2);
        \draw[edgex] (d4) -- (d1);
        \draw[edgex] (d4) -- (d3);
    \end{scope}

    %%%%%%%%%%%%%%%%%%%%%%%%%%%%%%%%%%%%%%
    % G^(5)
    %%%%%%%%%%%%%%%%%%%%%%%%%%%%%%%%%%%%%%
    \begin{scope}[xshift=18.0cm, yshift=0cm]
        \node[label] at (-1.2,1.75) {$G^{(5)}$:};

        \node[vertex] (e1) at (-0.6,1.2) {1};
        \node[vertex] (e3) at (0,0) {3};
        \node[vertex] (e2) at (-1.2,0) {2};
        \node[vertex] (e4) at (1.2,0) {4};
        \node[vertex] (e5) at (-0.6,-1.039) {5};

        \draw[edge]  (e2) -- (e1);
        \draw[edge]  (e3) -- (e1);
        \draw[edge]  (e3) -- (e2);
        \draw[edge]  (e4) -- (e1);
        \draw[edge]  (e4) -- (e3);
        \draw[edgex] (e5) -- (e2);
        \draw[edgex] (e5) -- (e3);
    \end{scope}

    \end{tikzpicture}
    }
    \caption{A realization of $G^{(1)},\ldots,G^{(5)}$ in a URD with $m=2$. In each step, the dotted edges are those added when the new vertex is introduced.}
    \label{fig:urd_process}
\end{figure}

For $m=1$, Definition~\ref{def:dag} gives the random recursive tree (RRT). See Figure~\ref{fig:urd_process} for a realization of the URD with $m=2$ and $n=1,\ldots,5$. Since every edge is directed from the vertex with the larger label to the vertex with the smaller label, all directed paths are oriented towards vertex $1$. For a vertex $v\in[n]$, we write $\deg_n(v)$ for its in-degree in $G^{(n)}$. Throughout, we set
\begin{equation}
    \label{eq:theta}
    \theta=\theta(m)\coloneq\frac{m+1}{m}.
\end{equation}
Fix $k\in\N$. For each $n$, let $V_1,\ldots,V_k$ be distinct vertices sampled uniformly at random from $[n]$. Theorem 2.4 in~\cite{EL26} shows that, for every fixed $c\in(0,m+1)$, there exists $\alpha>0$ such that, uniformly over $d_1=d_1(n),\ldots,d_k=d_k(n)\in\N$ satisfying $d_v<c\log n$ for all $v\in[k]$,
\begin{equation}
    \label{eq:joint_degree_distribution}
    \prob\bigl(\deg_n(V_v)\geq d_v,\ v\in[k]\bigr) = \theta^{-\sum_{v=1}^k d_v}\bigl(1+o(n^{-\alpha})\bigr).
\end{equation}
Thus, a fixed proportion below the scale $(m+1)\log n$, the joint degree distribution is asymptotically multivariate geometric with success parameters $1/(m+1)$. The approximation in~\eqref{eq:joint_degree_distribution} also suggests that the expected maximum degree should be of size $\log_\theta n$, since
\begin{equation}
    \prob\bigl(\deg_n(V_1)\geq\log_\theta n\bigr)\approx n^{-1}.
\end{equation}
This size of the maximum degree is known, see e.g.~\cite{Devroye.Lu.1995}. The transition studied in the present paper occurs at the larger scale $(m+1)\log n$. Indeed, $\log_\theta n=(\log n)/\log(1+1/m)<(m+1)\log n$. Consequently, the threshold $(m+1)\log n$ does not describe the location of the maximum degree. Rather, it marks the point at which the geometric approximation for the degree distribution is no longer valid. We now give a heuristic explanation for the appearance of this threshold, based on the Kingman coalescent representation formally introduced in Section~\ref{subsec:coal_constr}. Starting with $n$ isolated vertices, the coalescent proceeds by repeatedly selecting $m+1$ roots (i.e., vertices with out-degree zero) uniformly at random, choosing one of them uniformly to \emph{lose} while the others \emph{win}, and adding directed edges from the root that loses to those that win. The loser ceases to be a root, while each winner gains one incoming edge. Consequently, the degree of a vertex increases with each selection it wins until its first loss, after which its degree remains unchanged. For a fixed vertex $v$, let $S_n(v)$ denote the number of times it is selected during the coalescent construction. We will see that
\begin{equation}
    \mathbb{E}[S_n(v)]=(m+1)\log n+\cO(1), \qquad \Var(S_n(v))=(m+1)\log n+\cO(1).
\end{equation}
Since a given selected vertex $v$ wins with probability $m/(m+1)=\theta^{-1}$, we therefore have for $d=d(n)$,
\begin{equation}
    \label{eq:selection_count_heuristic}
    \prob\bigl(\deg_n(V)\geq d\bigr) \approx \theta^{-d}\prob\bigl(S_n(v)\geq d\bigr),
\end{equation}
where $V$ denotes the vertex in the URD that corresponds to $v$ in the Kingman coalescent representation. This suggests three different regimes. If $d$ stays sufficiently below $(m+1)\log n$, then $\prob(S_n(v)\geq d)\to1$ as $n\to\infty$, and only the geometric factor remains. Near $(m+1)\log n$, CLT-type fluctuations of $S_n(v)$ contribute a non-trivial Gaussian tail probability. Sufficiently far above this threshold, having enough selections becomes increasingly unlikely, leading to an additional suppression of the degree distribution. While geometric asymptotics have already been established in~\cite{EL26} for degrees staying a fixed proportion below $(m+1)\log n$ (see~\eqref{eq:joint_degree_distribution}), the following theorem provides a unified description of the transition. For $x\in\R$, write $\overline\Phi(x)\coloneq\prob(M>x)$, where $M$ is a standard normal random variable.

\begin{theorem}
    \label{theo}
    Fix $k,m\in\N$. For each $n$, let $V_1,\ldots,V_k$ be distinct vertices sampled uniformly at random from $[n]$, and let $d_1=d_1(n),\ldots,d_k=d_k(n)\in\N_0$.
    \begin{enumerate}[label=\textup{(\alph*)}]
        \item\label{theo:a} \emph{General transition.} Suppose that, for some $x_1,\ldots,x_k\in[-\infty,\infty]$, as $n\to\infty$,
        \begin{equation}
            \frac{d_v-(m+1)\log n}{\sqrt{\log n}} \longrightarrow x_v, \qquad v\in[k].
        \end{equation}
        Then, using the convention $\overline\Phi(-\infty)=1$ and $\overline\Phi(\infty)=0$, as $n\to\infty$,
        \begin{equation}
            \prob\bigl(\deg_n(V_v)\geq d_v,\ v\in[k]\bigr) = \theta^{-\sum_{v=1}^k d_v} \left( \prod_{v=1}^k \overline\Phi\left(\frac{x_v}{\sqrt{m+1}}\right) +o(1) \right).
        \end{equation}
        \item\label{theo:b} \emph{Fixed-gap supercritical regime.} Suppose that there exists $\eps>0$ such that, for all sufficiently large $n$, there is at least one $v\in[k]$ for which $d_v\geq(m+1+\eps)\log n$. Then there exists $\alpha>0$ such that, as $n\to\infty$,
        \begin{equation}
            \prob\bigl(\deg_n(V_v)\geq d_v,\ v\in[k]\bigr) = o\left( n^{-\alpha} \theta^{-\sum_{v=1}^k d_v} \right).
        \end{equation}
    \end{enumerate}
\end{theorem}

Part~\ref{theo:a} describes the transition around the selection threshold $(m+1)\log n$. It extends the previously known geometric asymptotics in the fixed-gap subcritical regime to degrees more than $\sqrt{\log n}$ below $(m+1)\log n$ in asymptotic order, however with an $o(1)$ relative error instead of the polynomially decaying relative error obtained in~\cite{EL26}. In the critical window of width $\sqrt{\log n}$, the geometric factor remains unchanged, but is multiplied by the probability that the Gaussian fluctuation in the number of selections is large enough to accommodate the prescribed degree. If at least one degree threshold exceeds $(m+1)\log n$ by more than $\sqrt{\log n}$ in asymptotic order, the joint degree distribution becomes negligible relative to the geometric contribution. Mixed regimes, where the prescribed degrees of different vertices lie on different scales, are also addressed by the result. Part~\ref{theo:b} strengthens the supercritical estimate to a polynomial suppression factor whenever at least one degree threshold stays at least $\eps\log n$ above $(m+1)\log n$ for some fixed $\eps>0$. A natural direction for future work is to obtain precise asymptotics for the joint degree distribution in the moderate-deviation regimes, particularly above the selection threshold, where Theorem~\ref{theo} only establishes a suppression relative to the geometric contribution.

The proof uses the Kingman coalescent representation, which we introduce in Section~\ref{sec:king}, to separate the probability of the vertices being selected sufficiently often from the geometric cost associated with obtaining sufficiently many consecutive successes to increase the degrees of the vertices. The estimate in the fixed-gap supercritical regime follows from an upper bound in terms of the selection counts together with Bernstein's inequality, as shown in Section~\ref{sec:supercrit}. For the general transition, we combine this upper bound with a corresponding lower bound and an asymptotic decoupling of the relevant selection counts, reducing the problem to a multivariate central limit theorem in Section~\ref{sec:critical}.

\section{Kingman coalescent for URDs}
\label{sec:king}

We use a Kingman coalescent construction of a URD, obtained as a special case of the coalescent construction introduced in~\cite{EL26} for URDs with freezing. In the first subsection, we verify that this construction indeed yields a URD, and in the second, we derive a representation of vertex degrees in terms of the coalescent.

\subsection{Kingman coalescent construction}
\label{subsec:coal_constr}

We say that a vertex of a directed acyclic graph is a \emph{root} if it has out-degree zero. For a root $v$, we call the subgraph induced by $v$ and all vertices from which $v$ can be reached via directed paths the \emph{in-component} of $v$. Note that in-components of distinct roots need not be disjoint. For an in-component $t$, we denote its unique root by $r(t)$.

\begin{definition}
    \label{def:coalescent}
    Let $m,n\in\N$. The \emph{Kingman $(m,n)$-coalescent} $(G_n,\ldots,G_1)$ is the sequence of random directed acyclic graphs constructed as follows. Start with $G_n$ consisting of $n$ isolated vertices labeled $[n]$. For $i=n,\ldots,2$, the graph $G_i$ has $i$ roots, whose in-components, listed in lexicographic order of their increasingly ordered vertex sets, are denoted by $t_1^{(i)},\ldots,t_i^{(i)}$. Choose a subset $B_i\subseteq[i]$ of size $(m+1)\wedge i$ uniformly at random and, conditional on $B_i$, choose $b_i\in B_i$ uniformly at random. We obtain $G_{i-1}$ from $G_i$ by adding the edges
    \begin{equation}
        \bigl(r(t_{b_i}^{(i)}),r(t_j^{(i)})\bigr), \qquad j\in B_i\setminus\{b_i\}.
    \end{equation}
    The roots $r(t_j^{(i)})$, $j\in B_i$ are called \emph{selected} at step $i$. We say that $r(t_{b_i}^{(i)})$ \emph{loses} at step $i$, since it ceases to be a root, whereas the other selected roots \emph{win}, as they remain roots and each gain one incoming edge. The random choices at different steps are independent.
\end{definition}

\begin{figure}
    \centering
    \begin{tikzpicture}[
        vertex/.style={minimum size=11pt, inner sep=0pt, circle, draw},
        up/.style={->, >=stealth, draw, thick},
        down/.style={<-, >=stealth, draw, thick},
        upx/.style={->, >=stealth, draw, thick, dotted},
        downx/.style={<-, >=stealth, draw, thick, dotted},
        label/.style={inner sep=0pt},
        scale=0.95,
        transform shape
    ]

    %%%%%%%%%%%%%%%%%%%%%%%%%%%%%%%%%%%%%%
    % LEFT COLUMN: G5, G4, G3
    %%%%%%%%%%%%%%%%%%%%%%%%%%%%%%%%%%%%%%

    % G5
    \begin{scope}[xshift=0cm, yshift=0cm]
        \node[label] at (-3.1,0.35) {$G_5$:};

        \node[vertex] (g5c) at (-0.1,0.35) {3};
        \node[vertex] (g5b) [left=.65cm of g5c] {2};
        \node[vertex] (g5a) [left=.65cm of g5b] {1};
        \node[vertex] (g5d) [right=.65cm of g5c] {4};
        \node[vertex] (g5e) [right=.65cm of g5d] {5}
            edge[upx, in=60, out=120] (g5b)
            edge[upx, in=60, out=120] (g5a);
    \end{scope}

    % G4
    \begin{scope}[xshift=0cm, yshift=-2.45cm]
        \node[label] at (-3.1,0.35) {$G_4$:};

        \node[vertex] (g4c) at (0.5,0.35) {3};
        \node[vertex] (g4b) [left=.95cm of g4c] {2}
            edge[downx, in=120, out=60] (g4c);
        \node[vertex] (g4e) [below left=1.0cm and .38cm of g4b] {5}
            edge[up] (g4b);
        \node[vertex] (g4a) [left=.95cm of g4b] {1}
            edge[down] (g4e);
        \node[vertex] (g4d) [right=.95cm of g4c] {4}
            edge[downx, in=60, out=120] (g4c);
    \end{scope}

    % G3
    \begin{scope}[xshift=0cm, yshift=-5.0cm]
        \node[label] at (-3.1,0.35) {$G_3$:};

        \node[vertex] (g3b) at (-0.3,0.35) {2};
        \node[vertex] (g3d) [right=1.2cm of g3b] {4}
            edge[downx, in=60, out=120] (g3b);
        \node[vertex] (g3c) [below right=1.0cm and .48cm of g3b] {3}
            edge[up] (g3b)
            edge[up] (g3d);

        \node[vertex] (g3a) [left=1.2cm of g3b] {1}
            edge[downx, in=120, out=60] (g3b);
        \node[vertex] (g3e) [below right=1.0cm and .48cm of g3a] {5}
            edge[up] (g3a)
            edge[up] (g3b);
    \end{scope}

    %%%%%%%%%%%%%%%%%%%%%%%%%%%%%%%%%%%%%%
    % RIGHT COLUMN: G2, G1
    %%%%%%%%%%%%%%%%%%%%%%%%%%%%%%%%%%%%%%

    % G2
    \begin{scope}[xshift=6.0cm, yshift=-0.75cm]
        \node[label] at (-2.8,0.35) {$G_2$:};

        \node[vertex] (g2d) at (0.55,0.35) {4};
        \node[vertex] (g2c) [below=1.0cm of g2d] {3}
            edge[up] (g2d);

        \node[vertex] (g2a) [left=2.2cm of g2d] {1}
            edge[upx, in=120, out=60] (g2d);
        \node[vertex] (g2e) [below=1.0cm of g2a] {5}
            edge[up] (g2a);

        \node[vertex] (g2b) [left=1.0cm of g2c] {2}
            edge[up] (g2a)
            edge[up] (g2d)
            edge[down] (g2e)
            edge[down] (g2c);
    \end{scope}

    % G1
    \begin{scope}[xshift=6.0cm, yshift=-4.2cm]
        \node[label] at (-2.8,1.4) {$G_1$:};

        \node[vertex] (g1b) at (-0.7,0) {2};
        \node[vertex] (g1a) [left=1.35cm of g1b] {1}
            edge[down] (g1b);
        \node[vertex] (g1c) [right=1.35cm of g1b] {3}
            edge[up] (g1b);
        \node[vertex] (g1e) [below right=1.1cm and .65cm of g1a] {5}
            edge[up] (g1a)
            edge[up] (g1b);
        \node[vertex] (g1d) [above left=1.1cm and .65cm of g1b] {4}
            edge[down] (g1a)
            edge[down] (g1b)
            edge[down] (g1c);
    \end{scope}

    \end{tikzpicture}
    \caption{A realization of the Kingman $(2,5)$-coalescent. At each step, the dotted edges are those added in the transition to the next graph; their common tail is the losing root and their heads are the winning roots.}
    \label{fig:coalescent_process}
\end{figure}

See Figure~\ref{fig:coalescent_process} for a realization of the Kingman $(2,5)$-coalescent. Let $m,n\in\N$. To relate the Kingman $(m,n)$-coalescent to the URD, we relabel the vertices of the final graph $G_1$ according to the steps at which they cease to be roots. More precisely, a vertex that loses at step $i\in\{2,\ldots,n\}$ receives the new label $i$, while the unique root remaining in $G_1$ receives label $1$. For $v\in[n]$, we denote its new label by $\ell_n(v)$. Since every edge added at step $i$ is directed from the losing root to a root that loses at a later step, the labels decrease along every directed edge. The following proposition shows that the resulting graph has the law of a URD.

\begin{proposition}
	\label{prop:relabeled_coalescent}
	Let $m,n\in\N$, and let $\widetilde G_1$ denote $G_1$ in the Kingman $(m,n)$-coalescent after relabeling every vertex $v\in G_1$ by $\ell_n(v)$. Then
	\begin{equation}
		\widetilde G_1\overset{d}{=}G^{(n)}.
	\end{equation}
	Moreover, $(\ell_n(v))_{v\in[n]}$ is a uniformly random permutation of $[n]$ and is independent of $\widetilde G_1$.
\end{proposition}

\begin{proof}
	Fix a step $i\in\{2,\ldots,n\}$ in the Kingman $(m,n)$-coalescent. For any given root $u$ and any set $W$ of $((m+1)\wedge i)-1$ other roots, the probability that $u$ is the loser and precisely the roots in $W$ are the winners is
	\begin{equation}
		\binom{i}{(m+1)\wedge i}^{-1}\frac1{(m+1)\wedge i} = \frac1i\binom{i-1}{((m+1)\wedge i)-1}^{-1}.
	\end{equation}
	Thus the loser is uniform among the $i$ roots and, conditional on the loser, the $m\wedge(i-1)$ winners are uniform among the remaining $i-1$ roots. Consequently, the order in which the vertices lose is a uniformly random permutation. After relabeling, the vertex with label $i$ therefore chooses its $m\wedge(i-1)$ out-neighbours uniformly from $[i-1]$. Moreover, for every prescribed loss order and prescribed collection of winner sets, the probability is
	\begin{equation}
		\prod_{i=2}^n \frac1i \binom{i-1}{m\wedge(i-1)}^{-1} = \frac1{n!} \prod_{i=2}^n \binom{i-1}{m\wedge(i-1)}^{-1}.
	\end{equation}
	For each possible relabeled graph, there are exactly $n!$ corresponding loss orders, one for each permutation of the original vertex labels. Summing over these orders cancels the factor $1/n!$, yielding the law of $G^{(n)}$. The factorization also establishes the claimed independence of the relabeled graph and the loss order.
\end{proof}

As an immediate consequence, we can transfer statements about uniformly sampled vertices of the URD to fixed vertices of the coalescent.

\begin{corollary}
	\label{corollary:interchangeability_degrees}
	For $m,n\in\N$, let $G^{(n)}$ be a URD and $G_1$ be the final graph in the Kingman $(m,n)$-coalescent. Let $\sigma$ be a uniformly random permutation of $[n]$. Then
	\begin{equation}
		\bigl(\deg_{G_1}(v),\ell_n(v)\bigr)_{v\in[n]} \overset{d}{=} \bigl(\deg_n(\sigma(v)),\sigma(v)\bigr)_{v\in[n]}.
	\end{equation}
\end{corollary}

Consequently, Theorem~\ref{theo}, where we consider distinct vertices $(V_v)_{v\in[k]}$ selected uniformly at random from $[n]$, can be proved as a result for fixed vertices $1,\ldots,k$ in the Kingman $(m,n)$-coalescent. From now on, we write, by a slight abuse of notation, also $\deg_n(v)\coloneq\deg_{G_1}(v)$ for the in-degree of $v$ in the final coalescent graph.

\subsection{Degree representation via selection sets}
\label{subsec:selection_and_connection_sets}

For our further analysis, we introduce notation for the selections made during the coalescent and a representation of degrees based on these selections, partially following~\cite{Lodewijks.2023}. Let $m,n\in\N$. For $i,v\in[n]$, let $t^{(i)}(v)$ be the first in-component in $t_1^{(i)},\ldots,t_i^{(i)}$ from Definition~\ref{def:coalescent} that contains $v$, and we call it the in-component \emph{associated} with $v$. We define
\begin{equation}
	s_{v,i} \coloneq \ind\bigl\{ t^{(i)}(v)=t_j^{(i)} \text{ for some }j\in B_i \bigr\}.
\end{equation}
Thus $s_{v,i}=1$ precisely when the root of the in-component associated with $v$ is selected at step $i$, in which case we also say that $v$ is \emph{selected} at step $i$. For fixed $v\in[n]$, the random variables $(s_{v,i})_{i\in\{2,\ldots,n\}}$ are independent and
\begin{equation}
	\prob(s_{v,i}=1) = \frac{(m+1)\wedge i}{i}.
\end{equation}
Hence the same probability holds conditional on the entire coalescent history before step $i$, which also yields independence across steps. The \emph{selection set} of $v$ is
\begin{equation}
	\label{eq:definition_selection_set}
	\mathcal S_n(v)\coloneq\{i\in\{2,\ldots,n\}:s_{v,i}=1\}=\{i_{v,1}>\cdots>i_{v,S_n(v)}\},
\end{equation}
where $i_{v,1},\ldots,i_{v,S_n(v)}$ are its decreasingly ordered elements, and we write $S_n(v)\coloneq|\mathcal S_n(v)|$. We further define
\begin{equation}
	h_{v,i} \coloneq \ind\bigl\{ t^{(i)}(v)=t_{b_i}^{(i)} \bigr\},
\end{equation}
so that $h_{v,i}=1$ precisely when the root of the in-component associated with $v$ loses at step $i$. For fixed $v$, the variables $(h_{v,i})_{i\in\{2,\ldots,n\}}$ are independent Bernoulli random variables with
\begin{equation}
	\prob(h_{v,i}=1)=\frac{(m+1)\wedge i}{i}\frac{1}{(m+1)\wedge i}=\frac1i.
\end{equation}
Until the first step at which $h_{v,i}=1$, the vertex $v$ itself is the root of $t^{(i)}(v)$. Whenever this in-component is selected and does not lose, an incoming edge is added to $v$. At the first selection at which it loses, $v$ ceases to be a root and its degree no longer changes. It follows that
\begin{equation}
	\label{eq:express_degree_via_selection_set}
	\deg_n(v) = \max\bigl\{ d\in\{0,\ldots,S_n(v)\}: h_{v,i_{v,1}}=\cdots=h_{v,i_{v,d}}=0 \bigr\}.
\end{equation}
Thus the degree of $v$ is the length of its initial winning streak among its selection times.

\section{Proof in the fixed-gap supercritical regime}
\label{sec:supercrit}

To prove Theorem~\ref{theo}\ref{theo:b}, we consider the joint degree distribution of vertices $1,\ldots,k$ for $k\in\N$ in the Kingman coalescent in the fixed-gap supercritical regime of degrees. We start by establishing an upper bound on the degree distribution using the selection sets introduced in the last section. This result is a generalization of~\cite[Lemma 4.3]{Addario.Eslava.2018}.
\begin{lemma}
	\label{lemma:upper_bound_tail_vertex_degrees}
	Let $m,n\in\N$ and recall $\theta$ from~\eqref{eq:theta}. For any $k \in [n]$, $d_1,\ldots,d_k \in \{0,\ldots,n-1\}$,
	\begin{equation}
		\prob(\deg_n(v)\geq d_v, v \in [k]) \leq \theta^{-\sum_v d_v}\prob(S_n(v) \geq d_v, v\in[k]).
	\end{equation}
\end{lemma}
\begin{proof}
	Let $\mathcal{A}$ be the set of tuples $A = (A_1,\ldots,A_k)$ satisfying $A_v \subset \{2,\ldots,n\}$ and $|A_v|=d_v$ for all $v\in[k]$. For every $A \in \mathcal{A}$, let $D_A$ be the event that $S_n(v) \geq d_v$ and $\{i_{v,1},\ldots,i_{v,d_v}\} = A_v$ for all $v\in[k]$. In words, $D_A$ is the event that $v$ is selected at least $d_v$ times and the first $d_v$ selections are given by $A_v$ for all $v\in[k]$. By~\eqref{eq:express_degree_via_selection_set} and since $D_A \subset \{S_n(v) \geq d_v, v\in[k]\}$, we have
	\begin{equation}
		\label{eq:decompose_degree_into_coinflips}
		\begin{split}
			\{\deg_n(v) \geq d_v, v\in[k]\}\cap D_A &= \{h_{v,i_{v,1}}=\ldots=h_{v,i_{v,d_v}}=0 \text{ for all } v\in[k]\}\cap D_A\\
			&= \{v \text{ wins at steps } i_{v,1},\ldots,i_{v,d_v} \text{ for all }v\in[k]\}\cap D_A.
		\end{split}
	\end{equation}
	Let $A^* \in \mathcal{A}$ such that $A^*_1,\ldots,A^*_k$ are disjoint. Recall from Definition~\ref{def:coalescent} the independent random variables $b_i$ determining the loser in the index set $B_i$ of selected roots in step $i \in \{2,\ldots,n\}$. Conditionally on $D_{A^*}$, the probability of the event $\{v \text{ wins at steps } i_{v,1},\ldots,i_{v,d_v}\allowbreak\text{ for all }v\in[k]\}$ is bounded from above by
	\begin{equation}
		\label{eq:degree_upper_bound_conditional_DA_rhs}
		\left(\frac{m}{m+1}\right)^{\sum_v d_v}=\theta^{-\sum_v d_v},
	\end{equation}
	as we have, for $i \in \{2,...,n\}$ and $\ell \in B_i$ that
	\begin{equation}
		\prob(b_i\neq \ell\,|\,B_i) = 1-\frac{1}{(m+1)\land i} \leq \frac{m}{m+1}.
	\end{equation}
	On the other hand, for arbitrary $A \in \mathcal{A}$ not necessarily consisting of disjoint sets, we can establish the same upper bound. Indeed, given that $v_1, v_2,\ldots, v_r$ are all selected at some step $i\in \{m+1,\ldots,n\}$, the probability that all these vertices win equals
	\begin{equation}
		\max\left\{\frac{m-r+1}{m+1},0\right\}\leq \Big(\frac{m}{m+1}\Big)^r.
	\end{equation}
	This is a direct result of the inequality
	\begin{equation}
		\label{eq:product_sum_inequality}
		\prod_{i=1}^r (1-x_i)\geq 1-\sum_{i=1}^r x_i,
	\end{equation}
	with $x_i=1/(m+1)$ for all $i\in\{1,\ldots,r\}$. If we also consider steps $i\in\{2,\ldots, m\}$, the inequality follows in a similar way. To conclude, we deduce for arbitrary $A \in \mathcal{A}$ that
	\begin{equation}
		\label{eq:degree_upper_bound_conditional_DA}
		\prob(v\text{ wins at steps } i_{v,1},\ldots,i_{v,d_v} \text{ for all }v\in[k]\, |\, D_A) \leq \theta^{-\sum_v d_v}.
	\end{equation}
	As the events $(D_A, A \in \mathcal{A})$ are pairwise disjoint and $\bigcup_{A \in \mathcal{A}}D_A = \{S_n(v) \geq d_v, v\in[k]\}$, it follows using~\eqref{eq:degree_upper_bound_conditional_DA}, as desired, that
	\begin{equation}
        \begin{split}
            \prob(\deg_n(v)\geq d_v, v \in [k]) &= \sum_{A\in\mathcal{A}}\prob(D_A,\deg_n(v)\geq d_v, v \in [k])\\
            &\leq  \sum_{A\in\mathcal{A}}\theta^{-\sum_v d_v}\prob(D_A)= \theta^{-\sum_v d_v}\prob(S_n(v) \geq d_v, v\in[k]).\qedhere
        \end{split}
	\end{equation}
\end{proof}

To make use of Lemma~\ref{lemma:upper_bound_tail_vertex_degrees} for proving Theorem~\ref{theo}\ref{theo:b}, we need a tail bound on the event $\{S_n(v)\geq d_v\text{ for all }v\in[k]\}$ in the fixed-gap supercritical regime of $d_v$ for $v\in[k]$, which is provided by the following lemma.

\begin{lemma}
    \label{lemma:sel_set_bernst}
    Fix $k,m\in\mathbb N$ and let $d_1=d_1(n),\ldots,d_k=d_k(n)\in\N_0$. Suppose that there exists $\eps>0$ such that, for all sufficiently large $n$, there is at least one $v\in[k]$ for which $d_v\geq (m+1+\eps)\log n$. Then there exists $\beta>0$ such that, as $n\to\infty$,
    \begin{equation}
        \mathbb P\bigl(S_n(v)\geq d_v,\ v\in[k]\bigr) =o(n^{-\beta}).
    \end{equation}
\end{lemma}
\begin{proof}
    For every sufficiently large $n$, choose $v_n\in[k]$ such that $d_{v_n}\geq (m+1+\eps)\log n$. Then
    \begin{align*}
        \mathbb P\bigl(S_n(v)\geq d_v,\ v\in[k]\bigr) &\leq \mathbb P\bigl(S_n(v_n)\geq d_{v_n}\bigr)\\
        &\leq \mathbb P\bigl(S_n(v_n)\geq (m+1+\eps)\log n\bigr)\\
        &= \mathbb P\bigl(S_n(1)\geq (m+1+\eps)\log n\bigr),
    \end{align*}
    where the last equality follows since the variables $S_n(v)$ are identically distributed. Since
    \begin{equation}
        \mathbb E[S_n(1)] =\sum_{i=2}^n\frac{(m+1)\wedge i}{i} =(m+1)\log n+\cO(1),
    \end{equation}
    set $a_n:=(m+1+\eps)\log n-\mathbb E[S_n(1)]=\eps\log n+\cO(1)$. Bernstein's inequality~\cite[Theorem~2.8]{Janson.Luczak.Rucinski.2000} yields
    \begin{equation}
        \mathbb P\bigl( S_n(1)\geq\mathbb E[S_n(1)]+a_n \bigr) \leq \exp\left\{ -\frac{a_n^2} {2\bigl(\mathbb E[S_n(1)]+a_n/3\bigr)} \right\} = n^{-\frac{\eps^2} {2(m+1+\eps/3)}+o(1)}.
    \end{equation}
    Hence the claim holds for any $0<\beta<\eps^2/(2(m+1+\eps/3))$.
\end{proof}

\begin{proof}[Proof of Theorem~\ref{theo}\ref{theo:b}]
    By Lemma~\ref{lemma:upper_bound_tail_vertex_degrees}, we have
    \begin{equation}
        \mathbb P\bigl( \deg_n(v)\geq d_v,\ v\in[k] \bigr) \leq \theta^{-\sum_{v=1}^k d_v} \mathbb P\bigl( S_n(v)\geq d_v,\ v\in[k] \bigr).
    \end{equation}
    By Lemma~\ref{lemma:sel_set_bernst}, there exists $\beta>0$ such that
    \begin{equation}
        \mathbb P\bigl( S_n(v)\geq d_v,\ v\in[k] \bigr) = o(n^{-\beta}).
    \end{equation}
    Consequently,
    \begin{equation}
        \mathbb P\bigl( \deg_n(v)\geq d_v,\ v\in[k] \bigr) = o\left( n^{-\beta} \theta^{-\sum_{v=1}^k d_v} \right).
    \end{equation}
    The result for the uniformly sampled distinct vertices $V_1,\ldots,V_k$ now follows from Corollary~\ref{corollary:interchangeability_degrees}.
\end{proof}

\section{Proof of the general transition}
\label{sec:critical}

To prove Theorem~\ref{theo}\ref{theo:a}, we need additionally to Lemma~\ref{lemma:upper_bound_tail_vertex_degrees} a matching lower bound. Here, to deal with correlations between the selection sets of multiple vertices, we introduce for $k\geq2$ the random variable
\begin{equation}
	\label{eq:definition_tau_k}
	\tau_k \coloneq \max\{2\leq i\leq n\colon s_{v,i}=s_{v^\prime,i}=1 \text{ for distinct }v,v^\prime\in[k]\},
\end{equation}
and set $\tau_1\coloneq1$. Thus, no two vertices in $[k]$ are selected simultaneously at any step $i>\tau_k$. By \cite[Lemma 4.5]{EL26}, for any integer-valued sequence $(t_n)_{n\in\N}$ diverging to infinity, there exists a constant $\zeta\in(0,1)$ such that
\begin{equation}
	\label{eq:tight_tauk}
	\prob(\tau_k\geq t_n)=o(t_n^{-\zeta}),
\end{equation}
which, in particular, implies that $(\tau_k)_{n\in\N}$ is a tight sequence of random variables. For natural numbers $t_n\leq n$, we define the \emph{truncated selection sets}
\begin{equation}
	\cS_{n,1}(v) \coloneq \{i\in\Omega_1\colon s_{v,i}=1\}\text{ for }v\in[k], \qquad \Omega_1\coloneq\{t_n,\ldots,n\},
\end{equation}
and write $\overline{\cS_{n,1}}\coloneq(\cS_{n,1}(v))_{v\in[k]}$. We refer to $(t_n)_{n\in\N}$ as the \emph{truncation sequence}. By \eqref{eq:tight_tauk}, the truncated selection sets are pairwise disjoint with probability $1-o(1)$ when $t_n\to\infty$ as $n\to\infty$. We now collect a couple of results on the truncated selection sets. It suffices for our purposes to fix the truncation sequence
\begin{equation}
    \label{eq:tn}
    t_n\coloneq\ceil{(\log n)^2},
\end{equation}
even though the following results hold in more generality. Lemma~\ref{lemma:lower_bound_tail_vertex_degrees_multiple} imposes no restriction on the growth of $t_n$. By contrast, Lemma~\ref{lemma:truncated_selection_sets_decouple_label_degree_whp} requires $t_n$ to grow at most polynomially, Lemma~\ref{lemma:replace_selection_sets_by_independent_copies} requires it to grow faster than logarithmically, and Proposition~\ref{prop:joint_selection_sets_critical} requires both conditions. The following lower bound is a special case of \cite[Lemma 4.2]{EL26} and allows us to restrict the degree event to the truncated selection sets.
\begin{lemma}
	\label{lemma:lower_bound_tail_vertex_degrees_multiple}
	Let $m,n\in\N$, $k\in[n]$, and $d_1,\ldots,d_k\in\{0,\ldots,n-1\}$. Then,
	\begin{equation}
		\prob(\deg_n(v)\geq d_v,\ v\in[k]) \geq \theta^{-\sum_v d_v} \prob\bigl( \tau_k<t_n, |\cS_{n,1}(v)|\geq d_v, v\in[k] \bigr).
	\end{equation}
\end{lemma}

To describe typical realizations of the truncated selection sets, let $\overline{J}=(J_v)_{v\in[k]}$ with $J_v\in\cP(\Omega_1)$ and, for $\delta\in(0,m+1)$, define
\begin{equation}
	\cB_{n,\delta} \coloneq \bigl\{ \overline{J}\in{\cP(\Omega_1)}^k\colon (J_v)_{v\in[k]}\text{ pairwise disjoint and } \bigl||J_v|-(m+1)\log n\bigr| \leq\delta\log n,\ v\in[k] \bigr\}.
\end{equation}
The following lemma shows that these conditions are asymptotically satisfied with high probability. It is a special case of \cite[Lemma 6.4]{EL26}.

\begin{lemma}
	\label{lemma:truncated_selection_sets_decouple_label_degree_whp}
	Fix $k,m\in\N$ and $\delta\in(0,m+1)$. Then, as $n\to\infty$,
	\begin{equation}
		\prob(\overline{\cS_{n,1}}\in\cB_{n,\delta}) =1-o(1).
	\end{equation}
\end{lemma}

We next prove that the joint law of the truncated selection sets can, on $\cB_{n,\delta}$, be replaced asymptotically by independent copies. Let $\overline{\cR_{n,1}}\coloneq(\cR_{n,1}(1),\ldots,\cR_{n,1}(k))$ consist of $k$ independent copies of $\cS_{n,1}(1)$.

\begin{lemma}
	\label{lemma:replace_selection_sets_by_independent_copies}
	Fix $k,m\in\N$, $\delta\in(0,m+1)$. Uniformly over $\overline{J}\in\cB_{n,\delta}$, as $n\to\infty$,
	\begin{equation}
		\prob(\overline{\cS_{n,1}}=\overline{J}) = \prob(\overline{\cR_{n,1}}=\overline{J})(1+o(1)).
	\end{equation}
\end{lemma}
\begin{proof}
    For $\overline J\in\cB_{n,\delta}$, set
    \begin{equation}
        \sigma_i\coloneq\sum_{v=1}^k\ind_{\{i\in J_v\}}, \qquad i\in\Omega_1.
    \end{equation}
    Since the sets $(J_v)_{v\in[k]}$ are pairwise disjoint, $\sigma_i\in\{0,1\}$. For $n$ sufficiently large,
    \begin{equation}
        \prob(\overline{\cS_{n,1}}=\overline J) = \prod_{\substack{i\in\Omega_1\\\sigma_i=0}} \frac{\binom{i-k}{m+1}}{\binom{i}{m+1}} \prod_{\substack{i\in\Omega_1\\\sigma_i=1}} \frac{\binom{i-k}{m}}{\binom{i}{m+1}}.
    \end{equation}
    Since $\overline{\cR_{n,1}}$ consists of $k$ independent copies of $\cS_{n,1}(1)$,
    \begin{equation}
        \prob(\overline{\cR_{n,1}}=\overline J) = \prod_{\substack{i\in\Omega_1\\\sigma_i=0}} \left(\frac{\binom{i-1}{m+1}}{\binom{i}{m+1}}\right)^k \prod_{\substack{i\in\Omega_1\\\sigma_i=1}} \frac{\binom{i-1}{m}}{\binom{i}{m+1}} \left(\frac{\binom{i-1}{m+1}}{\binom{i}{m+1}}\right)^{k-1}.
    \end{equation}
    For fixed $r\in\N$,
    \begin{equation}
        \frac{\binom{i-k}{r}}{\binom{i-1}{r}} =1-\frac{r(k-1)}{i}+\cO(i^{-2}), \qquad \frac{\binom{i}{r}}{\binom{i-1}{r}} =1+\frac{r}{i}+\cO(i^{-2}).
    \end{equation}
    Consequently,
    \begin{equation}
        \!\frac{\binom{i-k}{m+1}\binom{i}{m+1}^{k-1}} {\binom{i-1}{m+1}^k} = \bigl(1-\frac{(m+1)(k-1)}{i}+\cO(i^{-2})\bigr) \bigl(1+\frac{(m+1)(k-1)}{i}+\cO(i^{-2})\bigr)=1+\cO(i^{-2}),
    \end{equation}
    whereas
    \begin{equation}
        \frac{\binom{i-k}{m}\binom{i}{m+1}^{k-1}} {\binom{i-1}{m}\binom{i-1}{m+1}^{k-1}} = \bigl(1-\frac{m(k-1)}{i}+\cO(i^{-2})\bigr) \bigl(1+\frac{(m+1)(k-1)}{i}+\cO(i^{-2})\bigr)=1+\cO(i^{-1}).
    \end{equation}
    Thus,
    \begin{equation}
        \frac{\prob(\overline{\cS_{n,1}}=\overline J)} {\prob(\overline{\cR_{n,1}}=\overline J)} = \prod_{\substack{i\in\Omega_1\\\sigma_i=0}} \bigl(1+\cO(i^{-2})\bigr) \prod_{\substack{i\in\Omega_1\\\sigma_i=1}} \bigl(1+\cO(i^{-1})\bigr),
    \end{equation}
    where the constants depend only on $m$ and $k$. Since $\overline J\in\cB_{n,\delta}$,
    \begin{equation}
        \sum_{i\in\Omega_1}\sigma_i =\sum_{v=1}^k|J_v| \leq k(m+1+\delta)\log n.
    \end{equation}
    Hence, uniformly over $\overline J\in\cB_{n,\delta}$,
    \begin{equation}
        \sum_{\substack{i\in\Omega_1\\\sigma_i=0}}i^{-2} =\cO(t_n^{-1})=o(1), \qquad \sum_{\substack{i\in\Omega_1\\\sigma_i=1}}i^{-1} \leq\frac{k(m+1+\delta)\log n}{t_n} =\cO((\log n)^{-1})=o(1).
    \end{equation}
    Therefore,
    \begin{equation}
        \frac{\prob(\overline{\cS_{n,1}}=\overline J)} {\prob(\overline{\cR_{n,1}}=\overline J)} =\exp(o(1))=1+o(1)
    \end{equation}
    uniformly over $\overline J\in\cB_{n,\delta}$, which proves the claim.
\end{proof}

Equipped with Lemmas~\ref{lemma:truncated_selection_sets_decouple_label_degree_whp} and~\ref{lemma:replace_selection_sets_by_independent_copies}, we are able to decouple the asymptotic behavior of the selection sets of multiple vertices and prove the following multivariate central limit theorem. This provides the final ingredient for the proof of Theorem~\ref{theo}\ref{theo:a}.
\begin{proposition}
    \label{prop:joint_selection_sets_critical}
    Fix $k,m\in\N$, and let $M_1,\ldots,M_k$ be independent standard normal random variables. Then, as $n\to\infty$,
    \begin{equation}
    \label{eq:distrib_conv_sel_set}
        \left( \frac{|\cS_{n,1}(v)|-(m+1)\log n} {\sqrt{(m+1)\log n}}, \frac{S_n(v)-(m+1)\log n} {\sqrt{(m+1)\log n}} \right)_{v\in[k]} \toindis (M_v,M_v)_{v\in[k]}.
    \end{equation}
\end{proposition}
\begin{proof}
    Let $\overline{\cR_{n,1}}$ be as in Lemma~\ref{lemma:replace_selection_sets_by_independent_copies}. We first prove the corresponding convergence for $\overline{\cR_{n,1}}$. We have
    \begin{equation}
        |\cR_{n,1}(1)| \overset{d}{=} \sum_{i=t_n}^n s_{1,i}, \qquad s_{1,i}\sim\mathrm{Ber}\left(\frac{m+1}{i}\right),
    \end{equation}
    with independent summands, and hence by a simple bound on the harmonic sum
    \begin{equation}
        \mathbb E[|\cR_{n,1}(1)|]=(m+1)\log n+\cO(\log\log n), \qquad \Var(|\cR_{n,1}(1)|) =(m+1)\log n+\cO(\log\log n).
    \end{equation}
    Since the summands are bounded and the variance diverges, the Lindeberg-Feller theorem yields
    \begin{equation}
        \frac{|\cR_{n,1}(1)|-(m+1)\log n} {\sqrt{(m+1)\log n}} \toindis M_1.
    \end{equation}
    By independence, the corresponding $k$-dimensional convergence follows. We transfer this convergence to $\overline{\cS_{n,1}}$. Fix $\delta\in(0,m+1)$. By Lemma~\ref{lemma:truncated_selection_sets_decouple_label_degree_whp}, $\prob(\overline{\cS_{n,1}}\in\cB_{n,\delta})=1-o(1)$, while Lemma~\ref{lemma:replace_selection_sets_by_independent_copies} gives, uniformly for $\overline J\in\cB_{n,\delta}$, that $\prob(\overline{\cS_{n,1}}=\overline J) =\prob(\overline{\cR_{n,1}}=\overline J)(1+o(1))$. Here and below, the $o(1)$ terms may depend on $m,k$, and $\delta$, but are uniform in $\overline J$. Summing over $\cB_{n,\delta}$ shows that $\prob(\overline{\cR_{n,1}}\in\cB_{n,\delta})=1-o(1)$. For two discrete probability laws $P$ and $Q$, recall their total variation distance $d_{\mathrm{TV}}(P,Q)\coloneq \frac12\sum_x |P(x)-Q(x)|$. Therefore, again using Lemmas~\ref{lemma:truncated_selection_sets_decouple_label_degree_whp} and~\ref{lemma:replace_selection_sets_by_independent_copies},
    \begin{equation}
        2d_{\mathrm{TV}}\bigl(\mathbb P_{\overline{\cS_{n,1}}},\mathbb P_{\overline{\cR_{n,1}}}\bigr) = \sum_{\overline J\in\cB_{n,\delta}} \left|\prob(\overline{\cS_{n,1}}=\overline J)-\prob(\overline{\cR_{n,1}}=\overline J)\right|+o(1)=o(1).
    \end{equation}
    Since taking cardinalities is a measurable map, the total variation distance between the laws of the corresponding vectors of cardinalities also tends to zero. Hence
    \begin{equation}
        \left( \frac{|\cS_{n,1}(v)|-(m+1)\log n} {\sqrt{(m+1)\log n}} \right)_{v\in[k]} \toindis (M_v)_{v\in[k]}.
    \end{equation}
    Furthermore, for each $v\in[k]$,
    \begin{equation}
        0\leq S_n(v)-|\cS_{n,1}(v)| = |\cS_n(v)\cap\{2,\ldots,t_n-1\}|,
    \end{equation}
    whose expectation is $\cO(\log\log n)$. Since $k$ is fixed, a union bound and Markov's inequality give
    \begin{equation}
        \max_{v\in[k]} \frac{S_n(v)-|\cS_{n,1}(v)|}{\sqrt{\log n}} \toinp 0.
    \end{equation}
    The joint convergence in~\eqref{eq:distrib_conv_sel_set} now follows from Slutsky's theorem.
\end{proof}

\begin{proof}[Proof of Theorem~\ref{theo}\ref{theo:a}]
    Let $t_n\coloneq\ceil{(\log n)^2}$ as in~\eqref{eq:tn}, and set $a_v\coloneq x_v/\sqrt{m+1}$ for $v\in[k]$. By assumption, we have as $n\to\infty$,
    \begin{equation}
        \frac{d_v-(m+1)\log n}{\sqrt{(m+1)\log n}} \longrightarrow a_v\in[-\infty,\infty], \qquad v\in[k].
    \end{equation}
    We first show that
    \begin{equation}
        \label{eq:truncated_selection_critical_limit}
        \prob(S_n(v)\geq d_v,\ v\in[k]) \longrightarrow \prod_{v=1}^k\overline\Phi(a_v),\quad \prob(|\cS_{n,1}(v)|\geq d_v,\ v\in[k]) \longrightarrow \prod_{v=1}^k\overline\Phi(a_v),
    \end{equation}
    as $n\to\infty$. To this end, let $Z_{n,v}$ denote either of the normalized selection counts in Proposition~\ref{prop:joint_selection_sets_critical}, and write
    \[
        b_{n,v}\coloneq \frac{d_v-(m+1)\log n}{\sqrt{(m+1)\log n}}.
    \]
    If $a_v=\infty$, then the convergence in distribution by Proposition~\ref{prop:joint_selection_sets_critical} gives, as $n\to\infty$
    \begin{equation}
        \prob(Z_{n,v}\geq b_{n,v})\to0,
    \end{equation}
    and similarly, if $a_v=-\infty$, then $\prob(Z_{n,v}<b_{n,v})\to0$. Now let $I_+\coloneq\{v\in[k]\colon a_v=\infty\}$, $I_-\coloneq\{v\in[k]\colon a_v=-\infty\}$, and $I_0\coloneq\{v\in[k]\colon a_v\in\R\}$. If $I_+\neq\emptyset$, then
    \begin{equation}
        \prob(Z_{n,v}\geq b_{n,v},\ v\in[k]) \leq \min_{v\in I_+}\prob(Z_{n,v}\geq b_{n,v}) =o(1).
    \end{equation}
    Otherwise, by a union bound,
    \begin{equation}
        0\leq \prob(Z_{n,v}\geq b_{n,v},\ v\in I_0) -\prob(Z_{n,v}\geq b_{n,v},\ v\in[k])\leq \sum_{v\in I_-}\prob(Z_{n,v}<b_{n,v}) =o(1),
    \end{equation}
    and the joint convergence in Proposition~\ref{prop:joint_selection_sets_critical} yields
    \begin{equation}
        \prob(Z_{n,v}\geq b_{n,v},\ v\in I_0) \longrightarrow \prod_{v\in I_0}\overline\Phi(a_v).
    \end{equation}
    Thus, in both cases, \eqref{eq:truncated_selection_critical_limit} follows, with the convention that an empty product equals $1$.

    Moreover, by~\eqref{eq:tight_tauk}, $\prob(\tau_k\geq t_n)=o(1)$, and hence~\eqref{eq:truncated_selection_critical_limit} remains valid after intersecting the event $\{|\cS_{n,1}(v)|\geq d_v,\ v\in[k]\}$ with $\{\tau_k<t_n\}$. Applying Lemmas~\ref{lemma:upper_bound_tail_vertex_degrees} and~\ref{lemma:lower_bound_tail_vertex_degrees_multiple}, we obtain
    \begin{equation}
        \prob\bigl( \tau_k<t_n,|\cS_{n,1}(v)|\geq d_v,v\in[k] \bigr)\leq \theta^{\sum_v d_v} \prob(\deg_n(v)\geq d_v,v\in[k])\leq \prob(S_n(v)\geq d_v,v\in[k]).
    \end{equation}
    Since both outer probabilities converge to $\prod_{v=1}^k\overline\Phi(a_v)$, the claimed asymptotics follow. The result for the uniformly sampled distinct vertices $V_1,\ldots,V_k$ follows from Corollary~\ref{corollary:interchangeability_degrees}.
\end{proof}

\textbf{Acknowledgments.}
This work is based on the author's master thesis, supervised by Bas Lodewijks. We also thank Bas Lodewijks for valuable discussions and helpful comments on the manuscript.

\textbf{AI disclaimer.}
During the later stages of manuscript preparation, the author used AI tools for editing, proofreading, and improving the structure and clarity of the exposition. However, all mathematical proof ideas were developed by the author.

\printbibliography

\end{document}